\documentclass[10pt]{amsart}

\numberwithin{equation}{section}
\usepackage[a4paper, left=3.2cm, right=3.2cm]{geometry}
\usepackage{xypic}
\usepackage[english]{babel}
\usepackage{amsmath}
\usepackage{amsfonts}
\usepackage{amssymb}
\usepackage{amscd}
\usepackage{amsthm}
\usepackage{color}
\usepackage{xcolor}
\usepackage{verbatim}
\usepackage[colorlinks, linktocpage, citecolor = blue, linkcolor = blue]{hyperref}
\usepackage{tikz}	

\usetikzlibrary{matrix}					%geometric/algebraic description.
\usetikzlibrary{arrows, shapes, snakes, 
		       automata, backgrounds, 
		       petri, topaths}				
\newtheorem{theorem}{Theorem}[section]

\newtheorem{lemma}[theorem]{Lemma}
\newtheorem{proposition}[theorem]{Proposition}
\newtheorem{corollary}[theorem]{Corollary}

\theoremstyle{definition}

\newtheorem{remark}[theorem]{Remark}
\newtheorem*{acknowledgments}{Acknowledgments}

\theoremstyle{remark}

\usepackage{amscd, amssymb}
\newcommand\mynote[1]{\marginpar{\ \\ \small \tt #1}}
\newcommand\bel[1]{{\mynote{#1}}\begin{equation}\label{#1}}
\newcommand\mylabel[1]{\label{#1}}

\newcommand{\ZZ}{\mathbb{Z}}

\newcommand{\CC}{\mathbb{C}}

\newcommand{\GG}{\mathbb{G}}

\newcommand  {\shF}     {\mathcal{F}}

\newcommand  {\shO}     {\mathcal{O}}

\newcommand  {\Aut}     {\operatorname{Aut}}

\newcommand  {\Br}      {\operatorname{Br}}

\newcommand  {\et}      {{\text{\rm \'{e}t}}}

\newcommand  {\fet}      {{\text{\rm f\'{e}t}}}

\newcommand  {\Pro}    {{\text{Pro}}}
\newcommand  {\Fun}    {{\text{Fun}}}

\newcommand  {\Hom}     {\operatorname{Hom}}

\renewcommand  {\ker }  {\operatorname{Ker}}

\newcommand  {\lra}     {\longrightarrow}

\newcommand  {\Pic}     {\operatorname{Pic}}
\newcommand  {\PGL}     {\operatorname{PGL}}

\newcommand  {\ra}      {\rightarrow}
\newcommand  {\Ra}      {\Rightarrow}

\newcommand  {\Spec}    {\operatorname{Spec}}

\def\mydate{\number\day\space\ifcase\month \or January\or February\or March\or
April\or May\or June\or July\or
August\or September\or October\or November\or December\fi \space\number\year}
\begin{document}

\title[ On schemes with trivial higher \'etale homotopy groups]
      { On schemes with trivial higher \'etale homotopy groups}

\author[Mohammed Moutand]{Mohammed Moutand}
\address{ Moulay Ismail University, 
 Department of mathematics, 
Faculty of sciences,   B.P. 11201 Zitoune,  Mekn\`es,  Morocco.}
\curraddr{}
\email{m.moutand@edu.umi.ac.ma}
%\subjclass[2010]{14F22, 14F35}

%\subjclass{14F22,  20J06,  32J15}

%\dedicatory{Final version,  29 January 2020} 

\begin{abstract}

Let $(X,\bar x)$ be a pointed connected noetherian scheme. In this note,  we give characterizations  for the vanishing of the second \'etale homotopy group  $\pi^\et_2(X,\bar x)$ in terms of  splitting profinite-\'etale covers of $X$, and by means of universal   covering spaces of the   Artin-Mazur-Friedlander  \'etale homotopy type $Et(X)$.  In particular, this provides certain classes of schemes for which the Brauer map is surjective.
\end{abstract}

\maketitle

\section{Introduction}

Let $X$ be a  connected noetherian  scheme  with  a geometric base point $\bar x   \ra X$. Let   $\pi_2^\et(X, \bar x)$ be the  second \'etale homotopy group of $X$ as defined by Artin-Mazur \cite{Artin_Mazur:1969} and Friedlander \cite{Friedlander:1982} via the \'etale homotopy type $Et(X)$. In \cite{MTD} the author proved that if   $\pi_2^\et(X, \bar x)$ is trivial, then for   every locally constant constructible torsion \'etale sheaf  $\mathcal{F}$ on $X$,  and every class $\beta \in  H_{\et}^2(X,  \mathcal{F})$,  there exists a finite \'etale    cover $f: Y \rightarrow X$ such that $\beta_{|Y} = 0$ in  $H_{\et}^2(Y,  \shF_{|Y})$. When $\shF$ is the \'etale sheaf of $n$-th roots of unity $\mu_{n,X}$,  for  $n$ invertible in $X$, this gives in particular,  by a result of Gabber (Lemma \ref{galois}), a positive answer to  a  question raised by Grothendieck in  \cite{Grothendieck:1968} about the surjectivity of the Brauer map $\delta : \Br(X) \ra \Br'(X)$.  Where $\Br(X)$ is the Brauer group of Azumaya algebras on $X$, and $\Br'(X)$ is the cohomological Brauer group $ H_{\et}^2(X, \GG_{m,X})_{\rm tors}$.

In this note, we give characterizations for  the existence of the cover $f: Y \ra X$ in terms of  profinite-\'etale covers, which are defined by  Vakil and Wickelgren \cite{VW} as  inverse limits of finite \'etale covers of $X$, and by means of the universal covering space of the \'etale homotopy type $Et(X)$. When $X$ is geometrically unibranch, all conditions turn out to be equivalent. More precisely,  we prove the following main theorem.

\begin{theorem}\label{mainthm}
Let $X$ be a connected noetherian geometrically unibranch scheme with   a geometric base point   $\bar x \ra X$. Let $p: \tilde{Et}(X) \ra Et(X) $ be the  universal covering space of the \'etale homotopy type  $Et(X)$, and let $\hat f: \hat X \ra X$ be the pro-universal cover of X, that is $\hat{X}:= \varprojlim X_i$  is the inverse limit of  all finite \'etale  covers  $f_i: X_i \rightarrow X$ of $X$. The following statements  are equivalent
\begin{itemize}

\item [(i)] The second \'etale homotopy group $\pi^\et_2(X,\bar x)$ is trivial.

\item [(ii)] For  every locally constant constructible torsion \'etale sheaf  $\mathcal{F}$ on $X$  we have $$H_{\et}^2(X,  \mathcal{F}) \simeq   H^2(\pi_1^\et(X, \bar x),  \mathcal{F}_{\bar x}).$$

\item [(iii)] For every locally constant constructible torsion \'etale sheaf  $\mathcal{F}$ on $X$,  and every class $\beta \in  H_{\et}^2(X,  \mathcal{F})$   there exists an \'etale Galois   cover $g: Y \rightarrow X$ such that $\beta_{|Y} = 0$ in  $H_{\et}^2(Y,  \shF_{|Y})$.

\item [(iv)] For every locally constant constructible torsion \'etale sheaf  $\mathcal{F}$ on $X$,  and every class $\beta \in  H_{\et}^2(X,  \mathcal{F})$   there exists a finite \'etale    cover $f: Y \rightarrow X$ such that $\beta_{|Y} = 0$ in  $H_{\et}^2(Y,  \shF_{|Y})$.

\item [(v)] For  every locally constant constructible  \'etale sheaf  $\mathcal{F}$ on $X$,  the natural map 
$$H_\et^2(X,\shF) =  H^2(Et(X),\shF)   \ra  H^2(\tilde{Et}(X),p^*\shF).$$
is zero.

\item [(vi)]  For  every locally constant constructible torsion \'etale sheaf  $\mathcal{F}$ on $X$, $H_\et^2(\hat X,\shF_{|\hat X})=0$.

%If Moreover $X$ is geometrically unibranch then these conditions 	are equivalent to 

\item [(vii)] 
 For every locally constant constructible torsion \'etale sheaf  $\mathcal{F}$ on $X$,  and every class $\beta \in  H_{\et}^2(X,  \mathcal{F})$   there exists a  profinite-\'etale   cover $h: Y \rightarrow X$ such that $\beta_{|Y} = 0$ in  $H_{\et}^2(Y,  \shF_{|Y})$.

\end{itemize}
\end{theorem}

This is an algebraic analogous of well known topological statements, and  which  combines  some  related  results stated in   \cite{Ach17} and \cite{MTD} within   a  general setting. The geometrically unibranch assumption will only be used to prove the two implications (v)$\Rightarrow$(i)  and (vii)$\Rightarrow$(ii). It will be specifically required to interpret (using a $\infty$-categorical approach) the universal covering space of $Et(X)$  in terms of the \'etale homotopy type $Et(\hat X)$ of $\hat X$ (Proposition \ref{univetal}).

In addition to a  recent example of  Demarche and Szamuely  \cite{DS} for algebraic groups, to  which Theorem \ref{mainthm} could apply, we recover  some  results for smooth curves, abelian varieties, and  for the  class of Schr\"oer  spaces as defined by the author in  \cite{MTD}.

Using an homological  interpretation  of $\pi_2^\et(X,\bar x)$ by Pal \cite{PAL}, along with a recent result of  Lara-Srinivas-Stix  \cite{LSS}  on the topology of the \'etale fundamental group $\pi_1^\et(X,\bar x)$, we can prove the statement of Theorem \ref{mainthm} for proper schemes (see end of the last section).

\section{Preliminaries}
We shall introduce some preliminary notations in preparation for the proving of the main result of this
paper. 
\subsection{Etale homotopy type}

Following Artin-Mazur \cite{Artin_Mazur:1969} and Friedlander \cite{Friedlander:1982}, we consider the \'etale  homotopy type functor 
$$Et: Sch \lra Pro-Ho(Ssets)$$
which associates to any  locally noetherian scheme $X$, a pro-space $Et(X)$  in $Pro-Ho(Ssets)$ the pro-category of the homotopy category of simplicial sets, where the corresponding  cofiltered  category is $Ho(Hyp)$ the  homotopy category of rigid \'etale hypercoverings of $X$. 

For any abelian group $A$, we have a canonical isomorphism \cite[Corollary 9.3]{Artin_Mazur:1969} 
$$
H^n(Et(X),  A)  = H_{\et}^n(X, A).
$$

A given   geometric point $\bar x \ra X$ defines a point $ \bar x_{\et}$ on $Et(X)$,  hence one can define the \'etale homotopy pro-groups  for all $n \geq 0$:
$$
\pi_n^\et(X, \bar{x}):= \pi_n(Et(X),  \bar x_{\et}).
$$

In particular, by \cite[Corollary 10.7]{Artin_Mazur:1969}  $\pi_1^\et(X, \bar{x})$ is the usual Grothendieck \'etale fundamental group.
%\cite{Grothendieck_Raynaud:1971}.

\subsection{Remarks on locally constant constructible sheaves}

%Recall that an \'etale sheaf $\shF$ on $X$ is constructible if it has finite stalks and  for every closed immersion $i: Z \hookrightarrow   X$  with $Z$ irreducible, there exists a nonempty open subset  $U \subset Z$ such that $(i^*\shF)_{|U}$ is locally constant.

Let $\shF$ be a locally constant constructible \'etale  sheaf on $X$. By definition, $\shF$ has finite stalks and   there exists a covering $(Y_i \ra X)_{i\in I}$ in the \'etale topology such that for every $i \in I$, $\shF_{|Y_i}$ is constant. We have then the two following known results (see \cite[Chapter V, pages 155-156]{MIL}).

\begin{proposition}\label{reps}
Any   locally constant constructible \'etale  sheaf $\shF$  on a scheme $X$   is represented by a  finite \'etale group scheme $\tilde \shF$ over $X$, and there is a finite \'etale cover $f: Y \ra X$ such that $\shF_{|Y}$ is constant. 

\end{proposition}

%The following correspondence will play a fundamental role in the sequel

\begin{proposition}\label{corrspnd}
Let $X$ be a scheme, and $p$ a prime invertible in $X$.  Let $\shF$ be  a locally constant constructible  $p$-torsion \'etale  sheaf on $X$, and let $\bar x \ra X$ be a geometric point. The functor $\shF \ra \shF_{\bar x}$ induces an equivalence of categories between locally constant  $p$-torsion \'etale  sheaves with finite stalks and finite $\pi_1^\et(X,\bar x)$-modules of $p$-power order.
\end{proposition}

\begin{remark}\label{loclcnst}
It follows from Propositions \ref{reps} and \ref{corrspnd} that  for a given geometric point $\bar x \ra X$, and for  every locally constant constructible $p$-torsion  \'etale  sheaf $\shF$  on $X$, where $p$  is a prime invertible in $X$, there is a finite \'etale cover $g: Y \ra X$ such that
$$\shF_{\bar x} \simeq \shF_{|Y}  = \bigoplus_{i=1}^r (\ZZ/p^{n_i}\ZZ)^{m_i}$$
for some positive integer $r$.
\end{remark}
%In particular, for an integer $n$ invertible in $X$ we can apply this to the \'etale sheaf $\mu_{n,X}$.
\subsection{Truncated spaces and Truncated morphisms}
We  recall here some basic notions needed in the sequel  from Lurie's higher category theory. Our standard references are \cite{HTT}, and \cite[Appendix E]{SAG}.

 By space we mean a simplicial set. Let $S$ be the $\infty$-category  of pointed spaces. and  let $n \geq -2$:

- For $n \geq 0$; a pointed space $(X,x)$ is called $n$-truncated if the homotopy groups $\pi_k(X,x)$ vanish for all $k >  n$. By convention,  $X$  is (-1)-truncated if it is either empty or weakly contractible, and (-2)-truncated if $X$ is weakly contractible.

- A morphism of pointed spaces  $f: (Y,y) \ra (X,x)$ is called $n$-truncated if the fibers of $f$ are $n$-truncated, or equivalently if  the induced map of homotopy groups  $\pi_k(f):\pi_k(Y,y)  \ra \pi_k(X,x)$ is a monomorphism for $k=n+1$ and an  isomorphism for $k \geq n+2$.

- For any pointed space $(X,x)$ there exists   an initial object in the category of  0-truncated morphisms to $(X,x)$  called the universal covering space of $(X,x)$.

% and every pointed space admits an universal covering space.

Now in $\Pro(S)$ which is the $\infty$-category  of $\Pro$-spaces in $S$, all the aforementioned  notions  can be defined in a similar manner. Denote  by $S_{<\infty}$   the  $\infty$-category which is the full subcategory of $S$ whose object are truncated spaces, as discussed in \cite{Hoyflow}, there is  an equivalence of categories 
\begin{equation}\label{acclex}
\Pro(S_{<\infty}) \simeq \Fun^{acc,lex}(S_{<\infty},S)^{op}
\end{equation}
where the term on the right stands for the category of left exact accessible functors $F: S_{<\infty} \ra S$ (cf. \cite[Proposition A.8.1.6]{SAG}).

\section{Profinite-\'etale covers and universal covering spaces of $Et(X)$}

Let $X$ be a quasi-compact and quasi-separated scheme. Following Vakil and Wickelgren \cite{VW}, a cover $f : Y \ra X$ of $X$ is called profinite-\'etale if $Y$ is an inverse limit of finite \'etale covers of $X$. Such a cover exists as a  scheme  by    \cite[Exp VII,  5.1]{SGA4}.  The pro-universal cover $\hat f : \hat X \ra X$ of $X$  is by definition an initial object in the category of profinite-\'etale covers of $X$, thus  $\hat X=\varprojlim X_i$  is  the inverse limit of all finite \'etale covers of $X$. Note that $\hat X$ is the universal covering  of $X$  in the sense of \cite[Definition 6]{VW}. We begin this section by some  lifting and cohomological  properties of these types of covers.

%In the sens of \cite{} $Y$ is an pro-object in pro-$X_{\fet}$ the pro-finite \'etale site of $X$. 

\begin{proposition}\label{geompointuniv}\cite[Proposition 2.3]{VW}
\mylabel{vanish1}
Let  $X$   be a connected  scheme, and let $\bar x \ra X$ be a geometric point of $X$. If $f: Y \ra X$ is a profinite-\'etale cover of $X$, then there is a geometric point $\bar y \ra Y$ of $Y$ mapping to $\bar x$.

\end{proposition}

\begin{lemma}\label{hpro}
Let  $X$   be a connected noetherian scheme,  then for any locally constant  constructible torsion \'etale sheaf $\shF$ on $X$, we have  $H_{\et}^1(\hat{X}, \shF_{|\hat X}) = 0 $.

\end{lemma}

\begin{proof} 

We  prove that for every finite \'etale cover $f: Y \ra X$ and for every class $\beta \in  H_{\et}^1(Y,  \mathcal{F}_{|Y})$ there exists a finite \'etale cover $h: Z \ra X$ which factors through $f$, and such that  $\beta_{|Z} = 0$ in $ H_{\et}^1(Z,\mathcal{F}_{|Z})$. Choose  a geometric point $\bar x \ra X$, and let $f: (Y,\bar y)  \ra (X, \bar x) $  be a pointed  finite \'etale cover. We have by \cite[Lemma 4.10]{MTD}
$$H_{\et}^1(Y,  \mathcal{F}_{|Y}) \simeq   H^1(\pi_1^\et(Y, \bar y),  (\mathcal{F}_{|Y})_{\bar y}) \simeq  H_{\fet}^1(Y,  \mathcal{F}_{|Y}).$$
 Hence there exists a   finite \'etale cover $g: Z \ra Y$ such that $\beta_{|Z}= 0$ in $ H_{\et}^1(Z,  \mathcal{F}_{|Z})$. Therefore,  it suffices to take  $h = f \circ g: Z \ra X$ which is a finite \'etale cover of $X$.
\end{proof}

%The proposition provides the following strong vesrion of P\'al's theorem 

Let $X$ be a connected scheme, and let $Et(X)$ be its \'etale homotopy type. It is natural to ask if there exists a scheme $Y$ mapping to $X$, and for which $Et(Y)$ would be the universal covering space  $\tilde{Et} (X)$ of  $Et(X)$. We  present   here an affirmative answer (Proposition \ref{univetal}) when $X$ is at least noetherian and  geometrically unibranch,   by considering $Y=\hat X$ the pro-universal cover of $X$. The proof is due to  Hoyois and   uses $\infty$-categorical arguments.  We need first to state the  following two results which will be needed in the sequel.

\begin{proposition}\label{fintetalconn}\cite[Lemma 2.1]{SCHM} Any finite \'etale morphism of schemes $f : Y \ra X$ induces a connected covering space  $Et(f) : Et(Y) \ra Et(X)$.

\end{proposition}

\begin{corollary}\label{0tran}
Let  $f:  (Y,  \bar y) \ra (X,  \bar x)  $ be a finite \'etale  morphism of pointed connected schemes,  then the induced map
$$
\pi_n^\et(f): \pi_n^\et(Y, \bar y)    \lra   \pi_n^\et(X, \bar x) 
$$
is an isomorphism for every $n \geq 2$.

\end{corollary}

\begin{proposition}\cite{Hoyflow}\label{univetal}
Let $(X,\bar x)$ be a pointed  connected noetherian geometrically unibranch  scheme, then we have 
$$ \tilde{Et} (X)= Et(\hat X). $$
That is,   the  universal covering space of  the \'etale homotopy type  $Et(X)$ is  the \'etale homotopy type of the pro-universal cover $\hat X$.

\end{proposition}
\begin{proof} 

We regard $Et(X)$ as  an object in the $\infty$-category Pro(S). Since $X$ is   locally noetherian, it follows from \cite[Corollary 5.6]{Hoy18} that the   Artin-Mazur-Friedlander \'etale homotopy type functor $Et$ can be regarded, up to  protruncation (cf .\cite[1.2]{HHW}, see also \cite[Example 4.2.8]{BGH} or \cite[Example 1.9]{PJH} ), as a functor 
$$Et: Sch \lra \Pro(S_{<\infty}).$$

Therefore, under the equivalence (\ref{acclex}), $Et(X): S_{<\infty} \ra S$  would be the functor  sending a truncated space $K$ to the global section $H^0_\et(X,K)$  of the constant \'etale sheaf on $X$ with values in $K$. The statement of the proposition is equivalent then to the following three assertions:

\begin{itemize}
\item [(i)]     $Et(\hat X)$ is connected,

\item [(ii)]  The \'etale  homotopy group $\pi_1^\et(\hat X, \hat x)$ is trivial, where  $\hat x \ra \hat X$ be   the corresponding geometric point of $\hat X$ as in Proposition \ref{geompointuniv},

\item [(iii)]    $Et(\hat f): Et(\hat X) \ra  {Et} (X)$ is a 0-truncated map.
\end{itemize}

For the first assertion (i), since $X$ is quasi-compact and quasi-separated,  then any clopen subset of $\hat X$ comes from a clopen subset of some finite \'etale cover of $ X$, thus if $X$ is connected then so is $\hat X$. Assuming (iii), the second assertion (ii) follows from the fact that the homotopy pro-group $\pi_1^\et(X, \bar x)$,  and  consequently  $\pi_1^\et(\hat X, \hat x)$,  are profinite due to $X$ being geometrically unibranch. Hence $\pi_1^\et(\hat X, \hat x)=0$ since, by construction,  any finite \'etale cover of $\hat X$ is trivial. Now by Corollary \ref{0tran},  for any finite \'etale morphism  $f': X'  \ra X$  the induced map $Et(f') : Et(X') \ra Et(X)$ is 0-truncated.  This follows alternatively, as observed by Hoyois, from the fact that $Et(f)$ is the pullback of a morphism of groupo\"{i}ds  $\pi: \Xi' \ra \Xi$ with finite discrete fibers. Moreover,  it follows form  the equivalence (\ref{acclex})  and  \cite[Chapter III,  Lemma 1.16]{MIL} that for any truncated space $K$
$$Et(\hat X)(K)=H^0_\et(\hat X,K)= \varinjlim H^0_\et(X_i,K)=\varinjlim Et(X_i)(K).$$
This means that $Et$ preserves limits. Therefore, since 0-truncated morphisms are stable under limits,  the map  $Et(\hat f): Et(\hat X) \ra  \tilde{Et} (X)$ is  0-truncated. This proves (iii) and finishes the proof of the proposition.
\end{proof}

%The following corollary is an immediate consequence of Proposition \ref{univetal}, which generalizes the P\'al's interpretation \cite[Theorem 4.3]{PAL}  of the second \'etale homotopy group.
%\begin{corollary}
%For every connected noetherian geometrically unibranch  pointed scheme  $(X, \bar x)$, and for all $n \geq 0$  we have 
%$$ \pi_n^\et(X,\bar x)=H_n(\hat X, \ZZ).$$
%\end{corollary}

\section{Proof of the main result} 

In this section we give the proof of Theorem \ref{mainthm}. For   simplicity, throughout the proof, for  every locally constant constructible  torsion  \'etale  sheaf $\shF$  on $X$, we simply denote  $g:Z \ra X$  the finite \'etale cover of $X$ as in Remark \ref{loclcnst}. This means that  
$  \shF_{|Z} \simeq \shF_{\bar x} $ is constant.

(i)$\Rightarrow$(ii) Let $\shF$  be a locally constant constructible  torsion  \'etale  sheaf  on $X$. Consider the spectral sequence 
\begin{equation}\label{sphomotop}
E^{p,q}_2=H^p(\pi_1^\et(X, \bar x), H^q(\tilde{Et}(X) ,p^*\shF)) \Ra H_{\et}^{p+q}(X, \shF)
\end{equation}
associated to  the  universal covering  $p: \tilde{Et}(X) \ra Et(X)$. It gives  then an exact sequence (cf. \cite[Appendix B, page 309]{MIL})
\begin{align}\label{h12}
0 \ra  H^1(\pi_1^\et(X, \bar{x}),   H^0(\tilde{Et}(X),  p^*\shF))      \ra   H_{\et}^1(X,  \shF)  & \ra   H^0(\pi_1^\et(X, \bar{x}),   H^1(\tilde{Et}(X),   p^*\shF))   \\
 \ra  H^2(\pi_1^\et(X, \bar{x}),   H^0(\tilde{Et}(X), p^*\shF))     \ra   E^2_1 & \ra   H^1(\pi_1^\et(X, \bar{x}),   H^1(\tilde{Et}(X),  p^*\shF)) \notag
\end{align}
where 
$$
 E^2_1 =\ker(H_{\et}^2(X,  \shF) \ra H^0(\pi_1^\et(X, \bar{x}),   H^2(\tilde{Et}(X),  p^* \shF)).
$$
We have $H^0(\tilde{Et}(X),  p^*\shF) =  \shF_{\bar x}$ and $H^1(\tilde{Et}(X),  p^*\shF)  =0$. Moreover,  the universal coefficient theorem and  Hurewicz theorem  yield an isomorphism
$$
 H^2(\tilde{Et}(X),  p^*\shF)  \simeq  \Hom( \pi_2^\et(X, \bar{x}),  \shF_{\bar x}).
$$
Hence we get from the exact sequence \eqref{h12} an isomorphism 
$$
 H^2(\pi_1^\et(X, \bar{x}),   \shF_{\bar x}) \simeq \ker(H_{\et}^2(X,  \shF) \ra \Hom( \pi_2^\et(X, \bar{x}),  \shF_{\bar x})^{\pi_1^\et(X, \bar{x})}).
$$
Since $\pi_2^\et(X, \bar{x})=0$, the assertion follows immediately.

(ii)$\Rightarrow$(iii) As in the proof of  \cite[Theorem 4.2]{MTD}  we have 

\begin{align}
H_{\et}^2(X,  \mathcal{F}) & \simeq   H^2(\pi_1^\et(X, \bar{x}), \mathcal{F}_{\bar x}) \\
& \simeq \varinjlim H^2(\pi_1^\et(X, \bar{x})/ N,  \mathcal{F}_{\bar x}^N) \notag
\end{align}
where the limit runs over all normal open subgroups $N$ of $\pi_1^\et(X, \bar{x})$,  and 
 $\mathcal{F}_{\bar x}^N$ is the submodule of $N$-invariant elements.
Take  a class $\beta \in H_{\et}^2(X,  \mathcal{F})$,  it belongs  to a group $H^2(\pi_1^\et(X, \bar{x})/ N, \mathcal{F}_{\bar x}^N) $ for some open normal subgroup $N$. Therefore, according to the fundamental Galois correspondence, there exists a pointed \'etale Galois cover $g: (Y,  \bar y)  \rightarrow (X,  \bar x)$ with Galois group  $G:= \pi_1^\et(X, \bar{x})/ N$    and   $\pi_1^\et(Y, \bar{y}) = N $. This implies that $\beta_{|Y}=0$ in $  H_{\et}^2(Y, \mathcal{F}_{|Y})$.

(iii)$\Rightarrow$(iv) Obvious.

(iv)$\Rightarrow$(v) Any finite \'etale cover $f: Y \ra X$ induces by Proposition \ref{fintetalconn} a connected  covering space $Et(f) :  Et(Y) \ra  Et(X)$. The statement follows then from the fact that  the universal covering space $p: \tilde Et(X) \ra Et(X)$ factors through $Et(f)$. 

%$$H^2(\tilde{Et}(X),p^*\shF) \ra H^2(\tilde{Et}(X),\shF)$$
(v)$\Rightarrow$(i)   Let $G$ be a finite group. By applying  Proposition \ref{univetal}  and Remark \ref{loclcnst} to the constant \'etale sheaf with values in $G$, we can prove that $H^2(\tilde Et(X),G)=0$. Therefore, by the universal coefficients theorem for cohomology and Hurewicz theorem, we get $\Hom(\pi_2^\et(X, \bar{x}), G)=0$. Now since $\pi_2^\et(X, \bar{x})$ is profinite, it can be written as an inverse limit of finite groups $(G_i)_{i \in I}$. The assertion holds then by replacing $G$ by $G_i$ for each $i \in I$.

(vi)$\Rightarrow$(ii) 
Let $A = (X_i,  f_{ij})$  be the inverse  system of finite \'etale  covers of $X$. For every element $f_i: X_i \ra X$ in $A$,  there exists an \'etale  Galois cover $g_i: Y_i \ra X$ which factors through $f_i$,  hence elements in this inverse  system can be taken to be Galois. This being said, we get an identification (cf. \cite[Chapter I, page 40]{MIL})
$$\pi_1^\et(X, \bar{x})= \varprojlim \Aut_X(X_i)$$
where $ \Aut_X(X_i) $ is the group of $X$-automorphisms of $X_i$ acting on the right. For every \'etale Galois cover $g_i:  X_i \rightarrow X$,  consider the Hochschild-Serre spectral sequence 
$$
E^{p, q}_2(X_i)= H^p(\Aut_X(X_i),    H_{\et}^q(X_i,  \shF_{| X_i})) \Rightarrow  H_{\et}^{p+q}(X, \mathcal{F})
$$

Taking the direct  limit  of $E^{p, q}_2(X_i)$ we get by \cite[Chapter I, \S 2.2,  Proposition 8]{SER} 
\begin{align}
\varinjlim H^p(\Aut_X(X_i),   H_{\et}^q(X_i,  \shF_{| X_i}) ) & \simeq   H^p(\varprojlim \Aut_X(X_i),  \varinjlim H_{\et}^q(X_i,  \shF_{| X_i}) ) \notag \\
& \simeq  H^p(\pi_1^\et(X, \bar{x}),   H_{\et}^q(\hat{X},   \shF_{|\hat X}))\notag
\end{align}

Hence  we obtain  by \cite[Chapter III,  Remark 2.21.b]{MIL} a spectral sequence 
\begin{equation}\label{hspctral}
E^{p, q}_2 =  H^p(\pi_1^\et(X, \bar{x}),   H_{\et}^q(\hat{X},  \shF_{|\hat X}))   \Rightarrow  H_{\et}^{p+q}(X, \mathcal{F}).
\end{equation}

It yields  an exact sequence of low-degree terms
\begin{align}\label{sixterm}
0 \ra  H^1(\pi_1^\et(X, \bar{x}),   H_{\et}^0(\hat{X},  \shF_{|\hat X}))    \ra   H_{\et}^1(X, \mathcal{F})  & \ra   H^0(\pi_1^\et(X, \bar{x}),   H_{\et}^1(\hat{X}, \shF_{|\hat X}))\\
 \ra  H^2(\pi_1^\et(X, \bar{x}),   H_{\et}^0(\hat{X}, \shF_{|\hat X}))   \ra  E^2_1   & \ra   H^1(\pi_1^\et(X, \bar{x}),   H_{\et}^1(\hat{X},  \shF_{|\hat X})) \notag
\end{align}
where 
$$
 E^2_1 =\ker(H_{\et}^2(X,  \shF) \ra  H^0(\pi_1^\et(X, \bar{x}),   H_{\et}^2(\hat{X},  \shF_{|\hat X})).
$$

By Lemma \ref{hpro}  $ H_{\et}^1(\hat{X},  \shF_{|\hat X})=0$, and since $\hat f: \hat X \ra X$ factors through $g:Z \ra X$, we have 
$$ H_{\et}^0(\hat{X}, \shF_{|\hat X}) =  H_{\et}^0(Z, \shF_{|Z}) = \shF_{\bar x}.$$

Therefore, the assumption implies that 
$$H_{\et}^2(X,  \mathcal{F}) \simeq   H^2(\pi_1^\et(X, \bar x),  \mathcal{F}_{\bar x}).$$

(i)$\Rightarrow$(vi) We     prove that for every finite \'etale cover $f: Y \ra X$ and for every class $\beta \in  H_{\et}^2(Y,  \mathcal{F}_{|Y})$ there exists a finite \'etale cover $h: Y' \ra X$ which  factors through $Y$, and   such that $\beta_{|Y'} = 0$ in $ H_{\et}^2(Y',\mathcal{F}_{|Y'})$. Let $f: (Y,\bar y)  \ra (X, \bar x) $  be a pointed  finite \'etale cover. We have  by Corollary \ref{0tran}   $\pi_2^\et(Y, \bar y)=0$, hence for  a class  $\beta \in  H_{\et}^2(Y,  \mathcal{F}_{|Y})$,      there exists a   finite \'etale cover $f': Y' \ra Y$ such that $\beta_{|Y'} = 0$ in $ H_{\et}^2(Y',\mathcal{F}_{|Y'})$, thus as in the proof of Lemma \ref{hpro} we can take $h = f \circ f': Y' \ra X$.

  (iv)$\Rightarrow$(ii)    It suffices to prove that the term $E^2_1$  in the exact sequence (\ref{h12}) is equal to $ H_{\et}^2(X,  \shF)$. 
By the universal coefficient theorem  for the $\ZZ$-module $\shF_{\bar x}$, we have an isomorphism 
$$
 H^2(\tilde{Et}(X),  \mathcal{F}_{\bar x})  \simeq  \Hom( H_2( \tilde{Et}(X),  \ZZ),  \mathcal{F}_{\bar x}).
$$

Let $\alpha \in  H_{\et}^2(X,  \shF)$, and let $f :  Y	 \ra X$ be the finite \'etale cover that kills $\alpha$. Consider the following  diagram with commutative squares and triangles

\begin{equation}
\centerline{\xymatrix{  H_{\et}^2(Y,  \shF_{|Y}) \ar[d]  & H_{\et}^2(X,  \shF) \ar[d]  \ar[r] \ar[l]  \ar[dl] & \Hom(H_2( \tilde{Et}(X),  \ZZ),  \mathcal{F}_{\bar x})^{\pi_1^\et(X, \bar{x})} \ar@{^{(}->}[d] \\
 H_{\et}^2(Z',  \shF_{|Z'})    \ar[r] & H^2(\tilde{Et}(X),\mathcal{F}_{\bar x})   \ar[r]^-\simeq & \Hom(H_2( \tilde{Et}(X),  \ZZ),  \shF_{\bar x})  }}
\end{equation}
where $Z' = Y \times_X Z$. We have  $\shF_{|Z'} = \shF_{\bar x}$ and $ H_{\et}^2(Z',  \shF_{|Z'})  =  H^2(Et(Z'),  \shF_{\bar x})  $. A simple diagram chasing shows that the class $\alpha$ maps to zero in $ \Hom(H_2( \tilde{Et}(X),  \ZZ),  \mathcal{F}_{\bar x})^{\pi_1^\et(X, \bar{x})} $. Hence the assertion.

  (vi)$\Rightarrow$(vii) Take $Y=\hat X$.

  (vii)$\Rightarrow$(ii) Let $Y=\varprojlim Y_i$ be a profinite-\'etale cover of $X$. If $Z=Y_i$ for some $i$, then $\shF_{|Y} = \shF_{\bar x}$. In general, since inverse limits  commute with the fiber product, we can put $Z' =  Y \times_X Z = \varprojlim(Y_i \times_X Z) $. Every element $ Y_i \times_X Z$ is finite \'etale over $X$, hence $Z'$ is a profinite-\'etale cover of $X$, and we have $\shF_{|Z'} = \shF_{\bar x}$. On the other hand, we have by Proposition \ref{univetal} $\tilde{Et}(X) = Et(\hat X)$, this gives the following commutative diagram

\begin{equation}\label{digrm2}
\centerline{\xymatrix{   H_{\et}^2(Y,  \shF_{|Y}) \ar[d] & H_{\et}^2(X,  \shF) \ar[l] \ar[d]  \ar[r]   \ar[dl] & \Hom(H_2( \hat X ,  \ZZ),  \mathcal{F}_{\bar x})^{\pi_1^\et(X, \bar{x})} \ar@{^{(}->}[d] \\
 H_{\et}^2(Z',  \shF_{\bar x})    \ar[r] & H_\et^2(\hat X,\mathcal{F}_{\bar x})   \ar[r]^-\simeq & \Hom(H_2( \hat X,  \ZZ),  \shF_{\bar x})  }}
\end{equation}

By the property of the inverse limit, the pro-universal cover     factors through every  profinite-\'etale cover. Therefore, the assertion holds using the same argument as in the proof of    (iv)$\Rightarrow$(ii).

\begin{remark}\label{alln}
 The arguments sketched in the proof could be easily extended to prove that Theorem \ref{mainthm}  holds for every $n \geq 2$ in the sense that the assertion ``$\pi^\et_n(X, \bar x)=0$ for all $n \geq 2$''  is equivalent to each of the remaining assertions holding for every $n \geq 2$. This is done by the evaluation of the term $E_2^{n-1,0}$ of the long exact sequences of cohomology induced by  spectral sequences \eqref{sphomotop} and \eqref{hspctral} used in the proof. The statement for $n=2$ fixed,  clearly applies tanks to Lemma \ref{hpro}. This being said, one can naturally ask whether  the statement  is valid for a fixed $n \geq 3$.
\end{remark}

If $X$ is  defined over a field $k$, Theorem \ref{mainthm} behaves equivalently under base change:

\begin{proposition}
Let $X$ be a normal geometrically connected scheme of finite type over a field $k$, and let $K/k$ be an extension of fields. Then $X$ satisfies conditions of Theorem \ref{mainthm}  if and only if $X_K:= X \times_k  K$ does.
\end{proposition}

%An observation of  of Achinger theorem is stable under base change, 
\begin{proof}
%This follows by combining and the same  arguments of Achinger used in the proof of \cite[]{}. ]y replacing\ref{} by  Corollary \ref{} 

This is similar as \cite[Proposition 3.2.(c)]{Ach17}  with a minor change in the first part of  the proof: Notice first that since $X$ is normal, it is then geometrically unibranch. Now if $K$ is finite separable,  the natural map $p:X_K \ra X$ is finite \'etale and surjective, thus  $X$ is normal if only if $X_K$ is. The statement follows then by Corollary \ref{0tran}. If $K$ is the separable closure of $k$, then it can be written as a direct limit of finite separable field extensions of $k$. Therefore, the assertion follows by    \cite[Chapter III,  Lemma 1.16]{MIL} and Corollary \ref{0tran}. The rest of the proof is   the same as    in the proof of \cite[Proposition 3.2.(c)]{Ach17}.
\end{proof}

\section{Further comments  and Examples}

We present here some remarks and examples regarding the different frameworks resulting from Theorem \ref{mainthm}. We begin with its  implication  on Grothendieck's question concerning the surjectivity of the  Brauer map.

\subsection{Brauer groups and Brauer map}

 The Brauer group $\Br(X)$ of a scheme $X$ is  defined as the set of classes of Azumaya $\shO_X$-algebras on $X$ modulo similarity equivalence (see \cite[Chapter IV]{MIL}). The set of   Azumaya algebras  of rank $r^2$ is isomorphic to the \v Cech cohomology set $\check{H}_{\et}^1(X,   \PGL_r(\mathcal{O}_X))$. Non abelian cohomology gives  boundary maps  $$\check{H}_{\et}^1(X,   \PGL_r(\mathcal{O}_X))    \overset{\delta_r}\longrightarrow  \check{H}_{\et}^2(X, \mathbb{G}_{m,X})$$ which  induce an injective homomorphism of groups 
$$\delta: \Br(X) \lra \Br'(X)$$
called the Brauer map, where $\Br'(X) := H_{\et}^2(X, \GG_{m,X})_{\rm tors}$ is the torsion part of the \'etale  cohomology group  $H_{\et}^2(X, \GG_{m,X})$ (cf.  \cite[Proposition 1.4]{Grothendieck:1968}). For a quasi-compact scheme $X$, Grothendieck asked whether the map $\delta$ is surjective, that is when $\Br(X)= \Br'(X)$?  Affirmative answers on this question have been given for a very few cases, where the most important one which is proved by Gabber, is that of schemes with an ample invertible sheaf. Further, Gabber gave the following  technical criterion  to attack the problem (see \cite[ Lemma 1]{GBBR} and \cite[Remark 2.3]{MTD}).
%(see \cite{MTD} for a list of other cases
\begin{lemma}
\mylabel{galois}
 Let  $X$ be  a locally noetherian  scheme. Let $\beta \in  H_{\et}^2(X, \mu_{n,X})$ for some integer $n$. If there exists a finite locally free  cover $ f: Y \rightarrow X$ such that $\beta_{|Y}$=0 in $ H_{\et}^2(Y, \mu_{n,Y})$, then $ \beta$ lies in  $\Br(X)$.
\end{lemma}

Applying Theorem \ref{mainthm} to the \'etale sheaf $\shF=\mu_{n,X}$ for every $n$ invertible in $X$, we get then several classes of schemes for which  $\Br(X)= \Br'(X)$ up to $char(X)$-torsions. In particular, we will recover in the next subsections some results for previous solved cases.

\subsection{Smooth curves and abelian varieties}

%The construction of the pro-universals  covers of smooth curves and abelian varieties are explicitly  described by  

If $X$ is  a smooth curve  of genus $>0$ over a field $k$, then by  \cite[Proposition 3.6]{VW} we have  $H_\et^2(\hat X, \shF)=0$  for every  locally constant constructible torsion \'etale sheaf $\shF$ on $X$. Hence, Theorem \ref{mainthm} implies that $\Br(X)=\Br'(X)$ independently from the fact that $X$ is an algebraic $K(\pi,1)$ space as in \cite{MTD}.

Now for  an abelian variety $A$ over a separably closed field $k$, and regarding the definition of pro-finite \'etale covers in \cite{VW} as a relative spectrum, we have

$$\hat A= \underline{\Spec}(\bigoplus_{\alpha} \mathcal{L}^{-1}_{\alpha})$$ where the sum runs over the torsion elements $\alpha \in \Pic(A)$ and $\mathcal{L}_{\alpha}$ is the invertible sheaf corresponding to $\alpha$. For an arbitrary field $k$, the pro-universal cover $\hat A$ is determined by applying the above construction to $A_{\bar k}:= A \times_k \bar k$, where $\bar k$ is a separable closure of $k$. See  \cite[pages 507-508 ]{VW} for the explicit description of those two constructions.

Since abelian varieties have vanishing higher \'etale homotopy groups (see for instance the proof of \cite[Proposition 7.2]{MTD} ), it follows by Remark \ref{alln} that $H_\et^n(\hat A, \shF_{|\hat A})=0$ for every locally constant constructible torsion  \'etale sheaf $\shF$ on $A$ and for all $n \geq 1$. This extends the statement of \cite[Proposition 3.6]{VW} to a  class of higher dimensional varieties.

More generally, Demarche and Szamuely showed \cite[Theorem 1.1]{DS} that for any connected smooth algebraic group $G$ over a
separably closed field $k$ of characteristic $p \geq 0$, we have $\pi^\et_2(G^{{\wedge}(p')},\bar g)=0$, where $G^{{\wedge}(p')}$ is   the completion of the \'etale homotopy type $Et(G)$ with respect to
the class of finite groups of order prime-to-$p$, and the geometric point $\bar g \ra G$ is the unit element of $G(k)$.  In characteristic $0$, this implies in particular by Theorem \ref{mainthm}  that $\Br(G)=\Br'(G)$.

\subsection{Schr\"oer spaces}

It is proven by the author in \cite{MTD} that if $X$ is a Schr\"oer space, that is  a  connected scheme of finite type  over $\CC$ such that  its associated analytic space $X^{an}$ is a topological $K(\pi,  1)$ space, and the  topological fundamental group $\pi^{\rm top}_1(X^{an})$ is good in the sense of Serre (cf. \cite[Definition 3.4]{MTD}), then   $\Br(X)= \Br'(X).$  By the following proposition, one can alternatively apply  Theorem  \ref{mainthm} to get this result  in the case that $\pi_1^{\rm top}(X^{an})$ is finite, 

\begin{proposition}\label{complex}

Let $X$ be a  connected scheme  of finite type over $\CC$, such that $X^{an}$ is a $K(\pi,1)$ space with finite topological fundamental group $\pi_1^{\rm top}(X^{an})$, and let $\hat X$ be the pro-universal cover of $X$. Then  for every locally constant constructible torsion \'etale sheaf $\shF$ on $X$ and for all $n \geq 1$ the natural morphism
$$H_\et^n(X,\shF) \lra H_\et^n(\hat X,\shF_{|\hat X})$$
is zero.

\end{proposition}

\begin{proof}

By the Riemann existence theorem \cite[Chapter III,  Lemma 3.14]{MIL}, the functor $$F : (f: Z \ra X) \lra (f^{an}: Z^{an}  \ra X^{an})$$ induces an equivalence of the category of finite \'etale
covers of  $X$ with that of finite topological covers of $X^{an}$. Since $X$ is smooth, then $X^{an}$ is a complex manifold, hence the universal cover $p: C \ra X^{an}$ exists, and it is finite by assumption on $\pi_1^{\rm top}(X^{an})$ which is finite. Following an idea of  Thurston used in the proof of  \cite[ Proposition 13.2.4]{THMAN},  one can find an inverse system $(C_i,f_{ij})_{i \in I}$ of finite connected topological covers of $X^{an}$ with  $C= \varprojlim C_i $. 
Under the above equivalence,  covers $C$ and  $C_i$ correspond respectively to finite \'etale covers $g: Y \ra X$ and $g_i: Y_i \ra X$ with $C= Y^{an}$ and $C_i=Y_i^{an}$ for every $i \in I$. Let  $Y'=\varprojlim Y_i$,  and let $\shF$ be a locally constant constructible torsion \'etale sheaf on $X$. We obtain a natural commutative diagram

\centerline{\xymatrix{ Y^{an}  \ar[d]  \ar[rr]^{p}&  & X^{an} \ar[d] \\
 Y   \ar[r] & Y' \ar[r] &  X   \\
       &  \hat{X} \ar[ru] \ar[u] \ar[lu] &   }}

It gives a commutative diagram of \'etale cohomology groups for all $n \geq 1$

\centerline{\xymatrix{ H^n(X^{an}, \shF_{|X^{an}})   \ar[rr] &  & H^n(Y^{an},  \shF_{|Y^{an}})   \\
H^n_\et(X,  \shF) \ar[rd]  \ar[u]  \ar[r] & H^n_\et(Y',  \shF_{|Y'}) \ar[d] \ar[r] &  H^n_\et(Y,  \shF_{|Y})  \ar[ld]  \ar[u]  \\
       &  H^n_\et(\hat{X},  \shF_{|\hat{X}})    &   }}

Observe first that after replacing $Y$ by $Y''=Y \times_X Z$, where $g: Z \ra X $ is the finite \'etale cover of Remark \ref{loclcnst}, we may assume $\shF$  constant. Now since $X^{an}$ is a $K(\pi,1)$ space, it follows from the fibration  exact sequence of homotopy groups induced by the universal cover $p:  Y^{an} \ra  X^{an}$, that $ Y^{an} $ is weakly contractible, thus  by Hurewicz theorem we get  $\pi_n^{\rm top}(Y^{an})= H_n(Y^{an}, \ZZ)=0$ for all $n \geq 1$. Therefore, the  universal coefficient theorem would imply  that  $H^n(Y^{an},\shF_{|Y^{an}}) = 0$ for all $n \geq 1$. On the other hand, the two upper vertical maps are isomorphisms by Artin comparison theorem. The assertion follows by a simple diagram chasing.
\end{proof}

\subsection{The case of  proper schemes}  

Let $(X,  \bar x)$  be a pointed  smooth,  geometrically irreducible,  quasi-projective variety  over an algebraically closed  field $k$. Following P\'al \cite{PAL}, for every abelian group $A$ and $n \geq 0$ we  consider the \'etale homology groups
$$
H_n(X, A):= H_n(Et(X),  A).
$$

 We have  alternative proofs  for implications     (vii)$\Rightarrow$(ii) and  (v)$\Rightarrow$(i)  of  Theorem \ref{mainthm}. Indeed, for such a variety,  we have by  \cite[Theorem 4.3]{PAL}
$$
   \pi_2^\et(X, \bar{x})  \simeq  \varprojlim H_2( Et(X_i), \hat{ \ZZ})
$$
where  the limit runs over all finite \'etale  covers  $f_i: X_i \rightarrow X$ of $X$,
%If $X$ is not necessary geometrically unibranch then Theorem \eqref{mainthm} would hold under an additional condition on the \'etale fundamental group. 
and  $\hat  \ZZ =  \varprojlim \ZZ /n\ZZ$ is  the profinite completion of $\ZZ$.   Since $Et$ preserves limits, we have   by \cite[Chapter III,  Lemma 1.16]{MIL} 
$$
\pi_2^\et(X,\bar x) \simeq \varprojlim H_2(X_i, \hat{ \ZZ}) \simeq H_2(\hat{X}, \hat{ \ZZ}).
$$

In fact, this holds for any proper scheme over $k$. Indeed, by tracking  P\'al's proof of  \cite[Theorem 4.3]{PAL}, we see that  assumptions on $X$ are needed  to use \cite[Theorem 4.1]{PAL} which can be replaced by the general assertion  of Corollary \ref{0tran}, and they are used  for the \'etale fundamental group $\pi_1^\et(X,\bar x)$ to be topologically finitely generated. However,  Lara-Srinivas-Stix recently \cite{LSS} generalized this  statement by showing that the \'etale fundamental group of all connected proper schemes over $k$ is topologically finitely presented.

On the other hand, for  any  locally constant constructible torsion \'etale sheaf  $\mathcal{F}$ on $X$, the natural morphism $\ZZ \ra \hat \ZZ$ induces a natural structure of $\hat \ZZ$-module on the stalk  $\shF_{\bar x}$. Therefore, by the universal coefficient theorem, we get an isomorphism
 $$H_\et^2(\hat X,\mathcal{F}_{\bar x})   \simeq  \Hom(H_2( \hat X,  \hat \ZZ),  \shF_{\bar x}).$$
By replacing $\ZZ$ by $\hat \ZZ$ in the diagram (\ref{digrm2}), the implication (vii) $ \Rightarrow $ (ii)   follows by using the same argument.
Now for  (v) $ \Rightarrow $ (i), the universal coefficient theorem for cohomology and Lemma  \ref{hpro} show that $H_1(\hat X, \ZZ)= H_2(\hat X, \ZZ)=0$, and by the universal coefficient theorem for homology we get $H_2(\hat{X}, \hat{ \ZZ})=0$.

\begin{acknowledgments}I would like to thank Marc Hoyois from MathOverflow, the proof the Proposition \ref{univetal} is due to him. 
\end{acknowledgments}

\end{document}